\documentclass[11pt]{amsart}

\usepackage{graphicx}
\usepackage[labelsep=space]{caption}
\usepackage{amsfonts}
\usepackage{tikz-cd}
\usepackage{amsthm}
\usepackage{nccmath}
\usepackage{etex}
\usepackage{amssymb,mathtools}
\usepackage{mathrsfs}
\usepackage{dsfont}
\usepackage[all]{xy}
\usepackage{microtype}

\usepackage[utf8]{inputenc}
\usepackage[T1]{fontenc}
\usepackage{lmodern} 
\usepackage{latexsym}
\usepackage{MnSymbol}
\usepackage{nicefrac}
\usepackage{microtype}
\usepackage{color}
\usepackage{tikz-cd}
\usepackage{blindtext}
\usepackage{enumitem}
\usepackage{old-arrows}
\usepackage{empheq}
\usepackage{extpfeil}
\usepackage{hyperref}
\usepackage{theoremref}
\usepackage{float}
\usepackage{caption}
\usepackage{subcaption}

\makeatletter
\renewenvironment{proof}[1][\proofname]{%
	\par\pushQED{\qed}\normalfont%
	\topsep6\p@\@plus6\p@\relax
	\trivlist\item[\hskip\labelsep\bfseries#1\@addpunct{.}]%
	\ignorespaces
}{%
	\popQED\endtrivlist\@endpefalse
}
\makeatother

\usepackage{geometry}
\usepackage{secdot}
\usetikzlibrary{decorations.markings}
\tikzset{double line with arrow/.style args={#1,#2}{decorate,decoration={markings,%
			mark=at position 0 with {\coordinate (ta-base-1) at (0,1pt);
				\coordinate (ta-base-2) at (0,-1pt);},
			mark=at position 1 with {\d[#1] (ta-base-1) -- (0,1pt);
				\d[#2] (ta-base-2) -- (0,-1pt);
}}}}
\usepackage{ifxetex}
\usepackage{etoolbox}
\ifxetex
\usepackage{unicode-math}
\makeatletter
\patchcmd{\arrowfill@}{-7mu}{-14mu}{}{}
\patchcmd{\arrowfill@}{-7mu}{-14mu}{}{}
\patchcmd{\arrowfill@}{-2mu}{-4mu}{}{}
\patchcmd{\arrowfill@}{-2mu}{-4mu}{}{}
\makeatother
\fi

\theoremstyle{plain}
\newtheorem{theorem}{Theorem}[section]

\newtheorem{corollary}[theorem]{Corollary}
\newtheorem{proposition}[theorem]{Proposition}
\theoremstyle{remark}
\newtheorem{remark}[theorem]{Remark}
\theoremstyle{definition}

\usetikzlibrary{decorations.markings}
\tikzset{double line with arrow/.style args={#1,#2}{decorate,decoration={markings,%
			mark=at position 0 with {\coordinate (ta-base-1) at (0,1pt);
				\coordinate (ta-base-2) at (0,-1pt);},
			mark=at position 1 with {\draw[#1] (ta-base-1) -- (0,1pt);
				\draw[#2] (ta-base-2) -- (0,-1pt);
}}}}

\newtheorem{theoremx}{Theorem}

\begin{document}
	
\title{On sections of configurations of points on orientable surfaces}
\author{Stavroula Makri}
\address{Institut de Mathématiques de Toulouse, 118 route de Narbonne, Toulouse, France}
\email{stavroula.makri@univ-tlse3.fr}
	
	\begin{abstract}
		We study the configuration space of distinct, unordered points on compact orientable surfaces of genus $g$, denoted $S_g$. Specifically, we address the section problem, which concerns the addition of $n$ distinct points to an existing configuration of $m$ distinct points on $S_g$ in a way that ensures the new points vary continuously with respect to the initial configuration. This problem is equivalent to the splitting problem in surface braid groups.  With an algebraic approach, for $g\geq 1$ and $m\geq 2$, we establish a necessary condition for the existence of a section, showing that if a section exists, then $n$ must be a multiple of $m+(2g-2)$. 
	For $g\geq 1$ and $m=1$, we take a geometric approach to demonstrate that a section exists for all values of $n$.
	\end{abstract}
	\maketitle
	
	\section{Introduction}
In this paper we study the space of configurations of $n$ distinct unordered points
on orientable compact surfaces of genus $g\geq 0$. Let $\Sigma$ be a connected surface and $n, m\in \mathbb{N}$. We denote the $n^{th}$ ordered configuration space of the surface $\Sigma$ by
$$F_n(\Sigma)=\{(p_1,\dots,p_n)\in\Sigma^n\ | \ p_i\neq p_j \ \text{for all}\ i,j\in\{1,\dots,n\},\ i\neq j\}.$$
The $n^{th}$ unordered configuration space is the orbit space $UF_n(\Sigma):=F_n(\Sigma)/S_n$. Similarly, we can consider the space obtained by quotienting the $(n+m)^{th}$ configuration space of $\Sigma$ by the subgroup $S_n\times S_m$ of $S_{n+m}$, that is  $UF_{n,m}(\Sigma):=F_{n+m}(\Sigma)/S_n\times S_m$. 
In \cite{fadell1962configuration}, Fadell--Neuwirth proved that for $n,m\in\mathbb{N}$, where $1\leq m<n$, and for any connected surface $\Sigma$ with empty boundary, the map $$p: F_n(\Sigma)\to F_m(\Sigma)\ \text{given by}\ p(x_1,\dots,x_m,x_{m+1},\dots,x_n) = (x_1,\dots,x_m)$$ 
is a locally-trivial fibration. The fibre over a point $(x_1,\dots,x_m)$ of the base space $F_m(\Sigma)$ may be identified with the $(n-m)^{th}$ configuration space of $\Sigma$ with $m$ punctures, $F_{n-m}(\Sigma\setminus\{x_1,\dots,x_m\})$, which we interpret as a subspace of the total space $F_n(\Sigma)$ via the injective map $i: F_{n-m}(\Sigma\setminus\{x_1,\dots,x_m\})\to F_n(\Sigma)$, defined by $i(y_1,\dots,y_{n-m})= (x_1,\dots,x_m,y_1,\dots,y_{n-m})$.
Similarly, we consider the map $q: UF_{n,m}(\Sigma)\to UF_n(\Sigma)$, defined by forgetting the last $m$ coordinates. 
For any connected surface $\Sigma$ without boundary, the map $q$ is a locally-trivial fibration, whose fibre can be identified with the unordered configuration space $UF_m(\Sigma\setminus\{x_1,\dots,x_n\})$. 

An important question to understand is the space of sections of these fibrations. That is, the existence and the construction of continuous maps $s:F_m(\Sigma)\to F_n(\Sigma)$ and $\iota:UF_n(\Sigma)\to UF_{n,m}(\Sigma)$ satisfying $p\circ s=\text{id}_{F_m(\Sigma)}$ and $q\circ \iota=\text{id}_{UF_n(\Sigma)}$ respectively. A section can naturally be considered as the map that provides $m$ additional distinct points to a given configuration of $n$ distinct points that depends continuously on the position of the given $n$ points. This problem is pivotal in topology and geometric group theory. Numerous aspects of obstructing and classifying sections of bundles over configuration and moduli spaces have been explored, for example in \cite{MR4185935}, \cite{chen}, \cite{chenW}, \cite{kra} and \cite{hubbard}.

In \cite{fox1962braid} Fox--Neuwirth proved that
$$\pi_1(F_n(\Sigma))\cong P_n(\Sigma)\ \text{and}\ \pi_1(UF_n(\Sigma))\cong B_n(\Sigma),$$ where $P_n(\Sigma)$ is the pure braid groups of $\Sigma$ and $B_n(\Sigma)$ is the full braid groups of $\Sigma$.
Similarly, $\pi_1(UF_{n,m}(\Sigma))\cong B_{n,m}(\Sigma)$, where $B_{n,m}(\Sigma)$ is a subgroup of $B_{n+m}(\Sigma)$ known as a mixed braid group of $\Sigma$. From the long exact sequence in
homotopy of the fibration $p$ we obtain the Fadell--Neuwirth pure braid group short exact sequence: 
\begin{equation}\label{s1}
	\begin{tikzcd}
		1 \ar[r]& P_{n-m}(\Sigma\setminus\{x_1,\dots,x_m\}) \ar[r, "\bar{i}"] & P_n(\Sigma) \ar[r, "\bar{p}"] & P_m(\Sigma)\ar[r] & 1,
	\end{tikzcd}
\end{equation}
where $m\geq3$ when $\Sigma$ is the $2-$sphere $S^2$, $m\geq2$ when $\Sigma$ is the projective plane $\mathbb{R}P^2$, and where $\bar{i}$ and $\bar{p}_{n,m}$ are the homomorphisms induced by the maps $i$ and $p$ respectively. The homomorphism $\bar{p}:P_n(\Sigma)\to P_m(\Sigma)$ can be considered geometrically as the map that forgets the last $n-m$ strands of the braid. Similarly, from the long exact sequence in
homotopy of the fibration $q$ we obtain the so-called generalised Fadell--Neuwirth short exact sequence:
\begin{equation}\label{s2}
	\begin{tikzcd}
		1 \ar[r]& B_m(\Sigma\setminus\{x_1,\dots,x_n\}) \ar[r] & B_{n,m}(\Sigma) \ar[r, "\bar{q}"] & B_n(\Sigma)\ar[r] & 1,
	\end{tikzcd}
\end{equation}
where $n\geq3$ when $\Sigma=S^2$ and $n\geq2$ when $\Sigma=\mathbb{R}P^2$, and where $\bar{q}$ is the homomorphism induced by the map $q$. The homomorphism $\bar{q}:B_{n,m}(\Sigma)\to B_n(\Sigma)$ can be considered geometrically as the epimorphism that forgets the last $m$ strands of the braid. 

Thus, the section problem in the level of configuration spaces translates to an equally important problem in braid group theory, called the splitting problem. This problem refers to the question of whether or not the homomorphisms $\bar{p}$ and $\bar{q}$ admit a section, or equivalently, whether or not the short exact sequences \eqref{s1} and \eqref{s2} split, which has been a central question to the theory of surface braid groups. The existence of section of $\bar{p}$ was
studied  by Fadell \cite{fadellHomotopy}, Fadell--Neuwirth \cite{fadell1962configuration}, Fadell--Van Buskirk \cite{fadell1962braid}, Van Buskirk \cite{van1966braid} and Birman \cite{https://doi.org/10.1002/cpa.3160220104}, approaching it either geometrically or algebraically, and a complete solution was given by Gon\c{c}alves--Guaschi in \cite{gonccalves2010braid}. On the other hand, the splitting problem for generalised Fadell--Neuwirth short exact sequence \eqref{s2} does not yet have a complete solution. In particular, the only surface, besides the plane $\mathbb{R}^2$, for which this problem has been studied is the 2-sphere by Gon\c{c}alves--Guaschi \cite{gonccalves2005braid} and Chen--Salter \cite{MR4185935} and the projective plane by the author \cite{makri1}. It is important to note that in none of these cases a complete solution has been given, since the construction of such sections and even their simple existence is a challenging problem.
	
Let $S_g$ be an orientable compact surface of genus $g\geq 1$. In the first part of this work, we study the section problem for the space of unordered configurations of	$S_g$. In particular, first we approach the problem algebraically, studying the short exact sequence
	\begin{equation}\label{ms}
		\begin{tikzcd}
			1 \ar[r]& B_n(S_g\setminus\{x_1,\dots,x_m\}) \ar[r] & B_{n,m}(S_g) \ar[r, "\bar{q}"] & B_m(S_g)\ar[r] & 1,
		\end{tikzcd}
	\end{equation}
	that corresponds to the fibration $q: UF_{n,m}(S_g)\to UF_m(S_g)$. Using an algebraic method, that we describe in detail in Section \ref{S3}, which has been initially presented in \cite{gonccalves2010braid} and \cite{gonccalves2005braid}, we obtain, for $m\geq 2$, necessary conditions for the homomorphism $\bar{q}$ to admit a section. To do so, we obtain first a presentation for the mixed braid group $B_{n,m}(S_g)$ as well as for a certain quotient of that group, presented in Section \ref{S2}. For that, we use a presentation of $B_m(S_g)$ and $B_n(S_g\setminus\{x_1,\dots, x_m\})$, given in \cite{bellingeri2011exact}. To be compatible with the presentation of $B_n(S_g\setminus\{x_1,\dots, x_m\})$, $\bar{q}$ is considered to be the map that forgets the first $n$ strands of the braid and not the last $m$ strands. Moreover, for $m=1$, we construct, for all values of $n$, geometric sections for the fibration $q: UF_{n,1}(S_g)\to UF_1(S_g)$, which corresponds to the splitting of the short exact sequence \eqref{ms}. Thus, we obtain the following theorem which we prove in Section \ref{S3}.

	\begin{theoremx}\label{thm1}
		Let $g\geq 1$ and $n,m\geq1$.
		For $m=1$ the short exact sequence \eqref{ms} splits for all values of $n$ and $g$.
		For $m\geq 2$, if the short exact sequence \eqref{ms} splits, then $n=km + k(2g-2)$, for $k\in \mathbb{N}$.
	\end{theoremx}

\begin{remark}\label{equiv}
	Applying a standard argument given in \cite{gonccalves2005braid}, which tells us that the ﬁbration $q: UF_{n,m}(S_g)\to UF_m(S_g)$ admits a section if and only if the short exact sequence \eqref{ms} splits, we can reformulate our first result as follows: For $m\geq 2$, if the ﬁbration $q: UF_{n,m}(S_g)\to UF_m(S_g)$ admits a section, then $n=km + k(2g-2)$, for $k\in \mathbb{N}$.
\end{remark}

	\section{Presentation of $B_{n,m}(S_g)$, for $g\geq1$}\label{S2}
	Our convention is that throughout this text we read the elements of the braid groups from left to right.
	We denote by $S_g$ a closed orientable surface of genus $g\geq 1$. 
	In the following result we give a presentation of the braid groups $B_m(S_g)$, given by Bellingeri, which we use in order to determine a presentation of a certain quotient of $B_{n,m}(S_g)$.
	
	\begin{theorem}[Bellingeri, \cite{paolo}]\label{pres1}
		Let $m\geq 1$ and $S_g$ a closed orientable surface of genus $g\geq 1$. The following constitutes a presentation of $B_m(S_g)$.
		\\
		\\\textbf{Generators:} $\sigma_1, \dots, \sigma_{m-1}$, $a_1,\dots, a_g$, $b_1,\dots, b_g$.
		\\\textbf{Relations:} \begin{enumerate}
			\item[$(BR)$] $\sigma_i\sigma_j=\sigma_j\sigma_i$, for $|i-j|>1$,
			\item[] $\sigma_i\sigma_{i+1}\sigma_i = \sigma_{i+1}\sigma_i\sigma_{i+1}$, for $1\leq i\leq m-2,$
			\item[$(R_1)$] $a_r\sigma_i=\sigma_i a_r$, for $1\leq r\leq g$, $i\neq 1$,
			\item[] $b_r\sigma_i=\sigma_i b_r$, for $1\leq r\leq g$, $i\neq 1$,
			\item[$(R_2)$]$ \sigma_1^{-1}a_r\sigma_1^{-1}a_r=
			     a_r\sigma_1^{-1}a_r\sigma_1^{-1}$, for $1\leq r\leq g,$
			\item[] $\sigma_1^{-1}b_r\sigma_1^{-1}b_r=
			        b_r\sigma_1^{-1}b_r\sigma_1^{-1}$, for $1\leq r\leq g,$        
			\item[$(R_3)$]$ \sigma_1^{-1}a_s\sigma_1a_r=
			 a_r\sigma_1^{-1}a_s\sigma_1$, for $s < r,$
		    \item[] $\sigma_1^{-1}b_s\sigma_1b_r=
			 b_r\sigma_1^{-1}b_s\sigma_1$, for $s < r,$       
			 \item[] $\sigma_1^{-1}a_s\sigma_1b_r=
			 b_r\sigma_1^{-1}a_s\sigma_1$, for $s<r,$        
			 \item[] $\sigma_1^{-1}b_s\sigma_1a_r=
			 a_r\sigma_1^{-1}b_s\sigma_1$, for $s < r,$   
			 \item[$(R_4)$] $\sigma_1^{-1}a_r\sigma_1^{-1}b_r=b_r\sigma_1^{-1}a_r\sigma_1,$ for $1\leq r\leq g$,
			\item[$(SR)$] $[a_1, b_1^{-1}]\cdots[a_g, b_g^{-1}]=\sigma_1\sigma_2\cdots\sigma_{m-1}^{2}\cdots\sigma_2\sigma_1$, where $[a,b]=aba^{-1}b^{-1}$.
			        
		\end{enumerate}	
	\end{theorem}
	
	A surface $S_g$ can be represented as a polygon of $4g$ edges with the standard identification of edges. The generators $\sigma_i$, for $1\leq i\leq m-1$, are the classical braid generators on the disc. The generator $a_r$, for $1\leq r\leq g$, corresponds to the braid whose first strand goes through the edge $\alpha_r$ while the rest strands are trivial. Respectively, the generator $b_r$, for $1\leq r\leq g$, corresponds to the braid whose first strand goes through the edge $\beta_r$ while the rest strands are trivial. For this geometric representation of the generators see Figure \ref{fig1}.
		
	The aim of this section is to give a presentation of a certain quotient of the group $B_{n,m}(S_g)$.
	To do so, we will first give a presentation of $B_{n,m}(S_g)$ using the short exact sequence 
	$$1\rightarrow B_n(S_g\setminus\{x_1,\dots,x_m\})\rightarrow B_{n,m}(S_g)\xrightarrow{\bar{q}} B_m(S_g)\rightarrow 1,$$
	 and standard results about presentations of group extensions. The map $\bar{q}$ can be considered geometrically as forgetting the first $n$ strands of the braid. 	 
	 Thus, let us first state a presentation of $B_n(S_g\setminus\{x_1,\dots,x_m\})$.	
	
	\begin{theorem}[Bellingeri, \cite{paolo}]\label{pres2}
		Let $m, n\geq 1$. The following constitutes a presentation of $B_n(S_g\setminus\{x_1,\dots,x_m\})$.
	\\
	\\\textbf{Generators:} $\sigma_1, \dots, \sigma_{n-1}$, $a_1,\dots, a_g$, $b_1,\dots, b_g$, $z_1,\dots, z_{m-1}$.
	\\\textbf{Relations:} 
		$(BR)$, $(R_1)$, $(R_2)$, $(R_3)$, $(R_4)$ as in Theorem \ref{pres1},
		\begin{enumerate}
		\item[$(R_5)$] $z_j\sigma_i=\sigma_iz_j$, for $i\neq 1$, $1\leq j\leq m-1$,
		\item[$(R_6)$] $\sigma_1^{-1}z_j\sigma_1a_r=
		a_r\sigma_1^{-1}z_j\sigma_1$, for $1\leq r\leq g$, $1\leq j\leq m-1$,
		\item[] $\sigma_1^{-1}z_j\sigma_1b_r=
		b_r\sigma_1^{-1}z_j\sigma_1$, for $1\leq r\leq g$, $1\leq j\leq m-1$,
		\item[$(R_7)$] $\sigma_1^{-1}z_j\sigma_1z_k=
		z_k\sigma_1^{-1}z_j\sigma_1$, for $1\leq j\leq m-1$, $j<k$,    
		\item[$(R_8)$] $\sigma_1^{-1}z_j\sigma_1^{-1}z_j=
		z_j\sigma_1^{-1}z_j\sigma_1^{-1}$, for $1\leq j\leq m-1$.		
	\end{enumerate}	
	\end{theorem}
	
	The generators $\sigma_i$, for $1\leq i\leq n-1$, $a_r$ and $b_r$, for $1\leq r\leq g$, are the same as in Theorem \ref{pres1}. The generator $z_i$, for $1\leq i\leq m-1$, corresponds to the braid whose first strand wraps around the $i^{th}$ puncture while the rest strands are trivial. For this geometric representation of the generators see Figure \ref{fig1}.
	
	\begin{remark}
Note that a braid whose first strand wraps around the $m^{th}$ puncture with the rest strands trivial can be represented by the braid corresponding to the element $$z_m^{-1}=[a_1, b_1^{-1}]\cdots[a_g, b_g^{-1}]\sigma_1^{-1}\sigma_2^{-1}\cdots\sigma_{n-1}^{-2}\cdots \sigma_2^{-1}\sigma_1^{-1}z_1\cdots z_{m-1}.$$
	\end{remark}

For simplicity let us denote by $\beta_{n,m}$ the group $B_n(S_g\setminus\{x_1,\dots,x_m\})$ and by $\Gamma_2(\beta_{n,m})$ the commutator subgroup $[\beta_{n,m}, \beta_{n,m}]$ of $\beta_{n,m}$. From Theorem \ref{pres2} we can deduce a presentation of $\beta_{n,m}/\Gamma_2(\beta_{n,m})$, which we will use Section \ref{S3}.
	
	\begin{corollary}\label{quot1}
		For $n,m\geq 1$, the following constitutes a presentation of $\beta_{n,m}/\Gamma_2(\beta_{n,m})$.
		\textbf{Generators}: $\sigma$, $a_1,\dots ,a_g$, $b_1,\dots ,b_g$, $z_1, \dots, z_{m}$.
		\\ \textbf{Relations}:\begin{enumerate}
			\item[$(1)$] All generators commute pairwise,
			\item[$(2)$] $\sigma^2 = 1$,
			\item[$(3)$] $z_m=(z_1\cdots z_{m-1})^{-1}$.
		\end{enumerate}
		In particular, $$\beta_{n,m}/\Gamma_2(\beta_{n,m})\cong\mathbb{Z}^{2g+(m-1)} \times \mathbb{Z}_2,$$ where $a_1,\dots ,a_g$, $b_1,\dots ,b_g$, $z_1, \dots, z_{m-1}$ generate the $\mathbb{Z}^{2g+(m-1)}$-component and $\sigma$ the $\mathbb{Z}_2$-component.
	\end{corollary}

\begin{remark}
	For simplicity, in Corollary \ref{quot1} we denote by $\sigma$, $a_1,\dots, a_g$, $b_1,\dots, b_g$ and $z_1,\dots,z_m$ the coset representatives of the generators $\sigma_1,\dots,\sigma_{n-1}$, $a_1,\dots, a_g$, $b_1,\dots, b_g$ and $z_1,\dots,z_m$ of $\beta_{n,m}$, given by Theorem \ref{pres2}, in $\beta_{n,m}/\Gamma_2(\beta_{n,m})$. From relation $(BR)$ of Theorem \ref{pres1}, we get that all generators $\sigma_1,\dots,\sigma_{m-1}$ are mapped to the same element in $\beta_{n,m}/\Gamma_2(\beta_{n,m})$, which we denote by $\sigma$. From relation $(R_4)$ of Theorem \ref{pres1}, it follows that $\sigma^2=1$ in $\beta_{n,m}/\Gamma_2(\beta_{n,m})$. Moreover, note that in the generating set we have added the generator $z_m$, as we will use it in what follows, and at the same time we added $z_m=(z_1\cdots z_{m-1})^{-1}$ in the set of relations.
\end{remark}   	
	
	We give now a presentation of the group $B_{n,m}(S_g)$.
	
	\begin{theorem}\label{pres3}
		For $n,m\geq 1$, the following constitutes a presentation of $B_{n,m}(S_g)$.
		\\
		\\\textbf{Generators:} $\sigma_1,\dots,\sigma_{n-1},$
		$ a_1,\dots, a_g$, $b_1,\dots,b_g$, $z_1,\dots, z_{m-1},$\\
	    $\tau_1,\dots,\tau_{m-1}$, $c_1,\dots,c_g$, $d_1,\dots, d_g$.
		\\
		\\\textbf{Relations:} \begin{enumerate}	
			\item[$(I)$] Relations $(BR)$, $(R_1)$, $(R_2)$, $(R_3)$, $(R_4)$, $(R_5)$, $(R_6)$, $(R_7)$, $(R_8)$ of Theorem \ref{pres2}.
			\item[$(II)$] For $1\leq i,j \leq m-1$, $1\leq r,s\leq g$, \begin{enumerate}
			\item[$(\overline{BR})$] $\tau_i\tau_j=\tau_i\tau_j,$ for $|i-j|>1$, \\
			$\tau_i\tau_{i+1}\tau_i=\tau_{i+1}\tau_i\tau_{i+1},$ for $1\leq i < m-2,$
			\item[$(\overline{R_1})$] $c_r\tau_{i}=\tau_{i}c_r$, for $i\neq 1$,\\
			        $d_r\tau_{i}=\tau_{i}d_r$, for $i\neq 1$,
			\item[$(\overline{R_2})$] $\tau_1^{-1}c_r\tau_{1}^{-1}c_r=c_r\tau_1^{-1}c_r\tau_{1}^{-1}$,\\
			$\tau_1^{-1}d_r\tau_{1}^{-1}d_r=d_r\tau_1^{-1}d_r\tau_{1}^{-1}$,
			\item[$(\overline{R_3})$] $\tau_1^{-1}c_s\tau_1c_r=
			c_r\tau_1^{-1}c_s\tau_1$, for $s < r$,\\
			$\tau_1^{-1}c_s\tau_1d_r=
			d_r\tau_1^{-1}c_s\tau_1$, for $s < r$,\\       
	  	    $\tau_1^{-1}d_s\tau_1d_r=
		    d_r\tau_1^{-1}d_s\tau_1$, for $s < r$,\\      
			$\tau_1^{-1}d_s\tau_1c_r=
			c_r\tau_1^{-1}d_s\tau_1$, for $s < r$,
			\item[$(\overline{R_4})$] $\tau_1^{-1}c_r\tau_{1}^{-1}d_r=d_r\tau_1^{-1}c_r\tau_{1}$,
             \item[$(\overline{SR})$] $[c_1,d_1^{-1}]\cdots[c_g, d_g^{-1}]\tau_{1}^{-1}\tau_2^{-1}\cdots\tau_{m-1}^{-2}\cdots\tau_2^{-1}
             \tau_1^{-1}=\prod_{i=n-1}^{0}(\Sigma_iz_1^{-1}\Sigma_i^{-1})$,\\ 
             for $\Sigma_i=\sigma_i^{-1}\cdots\sigma_1^{-1}z_1\cdots z_m$ and $\Sigma_0=z_1\cdots z_m$,\\
             where $z_m=[a_1,b_1^{-1}]\cdots[a_g,b_g^{-1}]\sigma_1^{-1}\sigma_2^{-1}\cdots\sigma_{n-1}^{-2}\cdots\sigma_2^{-1}\sigma_1^{-1}z_1\cdots z_{m-1}$.
			                    \end{enumerate}
			\item[$(III)$] For $1\leq i \leq n-1$, $1\leq j, k\leq m-1$, $1\leq r,s\leq g$,
			\begin{enumerate}
			\item[$(a)$] $\tau_j\sigma_i\tau_j^{-1}=\sigma_i$,\\
				         $c_r\sigma_ic_r^{-1}=\sigma_i$,\\
				         $d_r\sigma_id_r^{-1}=\sigma_i$,
			\item[$(b)$] $\tau_ja_s\tau_j^{-1}=a_s$,
			\item []	\begin{equation*}
					c_ra_sc_r^{-1}=
					\begin{cases}
						(a_s^{-1}z_1^{-1})a_s(a_s^{-1}z_1^{-1})^{-1}, & \text{for}\ r=s,\\
						a_s, & \text{for} \ r<s,\\
						(a_r^{-1}z_1^{-1}a_rz_1)a_s(a_r^{-1}z_1^{-1}a_rz_1)^{-1}, & \text{for} \ r>s,
					\end{cases}
				\end{equation*}
		    \item[]\begin{equation*}
		    	d_ra_sd_r^{-1}=
		    	\begin{cases}
		    		(b_s^{-1}z_1^{-1}b_s)z_1a_s(b_s^{-1}z_1^{-1}b_s)^{-1}, & \text{for}\ r=s,\\
		    		a_s, & \text{for} \ r<s,\\
		    		(b_r^{-1}z_1^{-1}b_rz_1)a_s(b_r^{-1}z_1^{-1}b_rz_1)^{-1}, & \text{for} \ r>s,
		    	\end{cases}
		    \end{equation*}
			\item[$(c)$] $\tau_jb_s\tau_j^{-1}=b_s$,
			\item []	\begin{equation*}
				c_rb_sc_r^{-1}=
				\begin{cases}
					a_s^{-1}z_1^{-1}a_sb_s, & \text{for}\ r=s,\\
					b_s, & \text{for} \ r<s,\\
					(a_r^{-1}z_1^{-1}a_rz_1)b_s(a_r^{-1}z_1^{-1}a_rz_1)^{-1}, & \text{for} \ r>s,
				\end{cases}
			\end{equation*}
			\item[]\begin{equation*}
				d_rb_sd_r^{-1}=
				\begin{cases}
					(b_s^{-1}z_1^{-1})b_s(b_s^{-1}z_1^{-1})^{-1}, & \text{for}\ r=s,\\
					b_s, & \text{for} \ r<s,\\
					(b_r^{-1}z_1^{-1}b_rz_1)b_s(b_r^{-1}z_1^{-1}b_rz_1)^{-1}, & \text{for} \ r>s,
				\end{cases}
			\end{equation*}
			\item[$(d)$] \begin{equation*}
				\tau_jz_k\tau_j^{-1}=
			  	\begin{cases}
					z_k, & \text{for}\ j\neq k-1, k,\\
					z_{k-1}, & \text{for} \ j=k-1,\\
					z_kz_{k+1}z_k^{-1}, & \text{for} \ j=k,
				\end{cases}
			\end{equation*}
		\item[]	\begin{equation*}
				c_rz_kc_r^{-1}=
				\begin{cases}
					z_k, & \text{for}\ k\neq1,\\
					a_rz_ka_r^{-1}, & \text{for} \ k=1,
				\end{cases}
			\end{equation*}
		\item[]	\begin{equation*}
				d_rz_kd_r^{-1}=
				\begin{cases}
					z_k, & \text{for}\ k\neq 1,\\
					b_r^{-1}z_kb_r, & \text{for} \ k=1.
				\end{cases}
			\end{equation*}
			\end{enumerate}
		\end{enumerate}
		
	\end{theorem}
	Note that the generators $\tau_1,\dots,\tau_{m-1}$, $c_1,\dots, c_g$ and $d_1,\dots, d_g$ of $B_{n,m}(S_g)$ correspond to the standard generators $\sigma_i, \dots, \sigma_{m-1}$, $a_1,\dots, a_g$ and $b_1,\dots, b_g$ as described in Theorem \ref{pres1}, but with base point the last $m$ points, see Figure \ref{fig1}.
	
\begin{figure}[!h]
	\centering
	\includegraphics[width=0.4\textwidth]{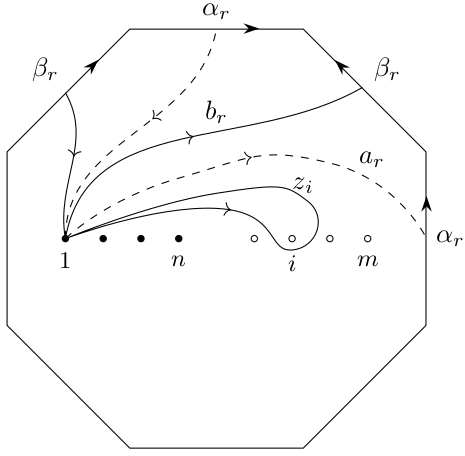}
		\includegraphics[width=0.4\textwidth]{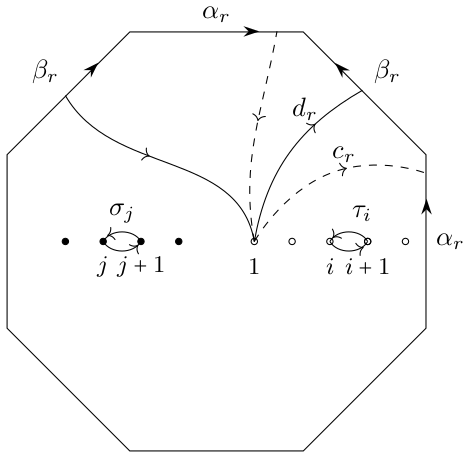}
	\caption{Geometric representation of the generators $a_r, b_r, z_i, \sigma_j, c_r, d_r$ and $ \tau_i $ of $B_{n,m}(S_g)$. Note that $a_r, b_r, z_i, \sigma_j$ are the generators that correspond to the first $n$ points while $c_r, d_r, \tau_i$ are the generators that correspond to the last $m$ points.}
	\label{fig1}
\end{figure}

	\begin{proof}
		In order to obtain a presentation of $B_{n,m}(S_g)$, we apply standard techniques for obtaining presentations of group extensions as described in [\cite{johnsonpresentation}, p. 139] based on the following short exact sequence:
		
		$$1\rightarrow \beta_{n,m}\rightarrow B_{n,m}(S_g)\xrightarrow{\bar{q}} B_m(S_g)\rightarrow 1,$$
		where the map $\bar{q}$ can be considered geometrically as the epimorphism that forgets the first $n$ strands.
		For simplicity, we denote by $\tau_1,\dots,\tau_{m-1}$, $c_1,\dots,c_g$, $d_1,\dots, d_g$ the corresponding coset representatives in $B_{n,m}(S_g)$. 
		The union of these elements together with the generators of $\beta_{n,m}$ of Theorem \ref{pres2} gives us the set of generators of $B_{n,m}(S_g)$.
		
		Based on [\cite{johnsonpresentation}, p. 139] we will obtain three classes of relations in $B_{n,m}(S_g)$. The first class of relations is the set of relations of $\ker(\bar{q})$, that is, of $\beta_{n,m}$, which are the relations in the class $(I)$ of the statement. 
		
		The second class of relations is obtained by rewriting the relations of $B_m(S_g)$, given in Theorem \ref{pres1}, in terms of the chosen coset representatives in $B_{n,m}(S_g)$ and expressing the corresponding elements as a word in $\beta_{n,m}$. Thus, we get relations $(\overline{BR}), (\overline{R_1}), (\overline{R_2}), (\overline{R_3}), (\overline{R_4})$ in the class $(II)$ of the statement. The relation $$[c_1, d_1^{-1}]\cdots[c_g, d_g^{-1}]=\tau_1\tau_2\cdots\tau_{m-1}^{2}\cdots\tau_2\tau_1$$ of $B_m(S_g)$ can be expressed as a word in $\beta_{n,m}$ as $$[c_1,d_1^{-1}]\cdots[c_g, d_g^{-1}]\tau_{1}^{-1}\tau_2^{-1}\cdots\tau_{m-1}^{-2}\cdots\tau_2^{-1}
		\tau_1^{-1}=\prod_{i=n-1}^{0}(\Sigma_iz_1^{-1}\Sigma_i^{-1}),\ \text{for}\ \Sigma_i=\sigma_i^{-1}\cdots\sigma_1^{-1}z_1\cdots z_m,$$
	 and $\Sigma_0=z_1\cdots z_m$, where $z_m=[a_1,b_1^{-1}]\cdots[a_g,b_g^{-1}]\sigma_1^{-1}\sigma_2^{-1}\cdots\sigma_{n-1}^{-2}\cdots\sigma_2^{-1}\sigma_1^{-1}z_1\cdots z_{m-1}$. 
		This relation is obtained geometrically using the geometric representation of the generators given in Figure \ref{fig1}, and it corresponds to relation $(\overline{SR})$ in the class $(II)$ of the statement.

		The third class of relations is derived from conjugating the generators of $\beta_{n,m}$ by the coset representatives $\tau_1,\dots,\tau_{m-1}$, $c_1,\dots,c_g$, $d_1,\dots, d_g$. The relations of this class are obtained geometrically. 
		\begin{itemize}
			\item By conjugating the generators $\sigma_i$, for $1\leq i\leq n-1$, we obtain relations $(III)(a)$ of the statement.
			\item By conjugating the generators $a_s$, for $1\leq s\leq g$, we have    
			relations $(III)(b)$ of the statement.
			\item By conjugating the generators $b_s$, for $1\leq s\leq g$, we have    
			relations $(III)(c)$ of the statement.
			\item By conjugating the generators $z_k$, for $1\leq k\leq m-1$, we have    
			relations $(III)(d)$ of the statement.
		\end{itemize}
			
		As a result, we have obtained a generating set and a complete set of relations, which coincide with those given in the statement, and this completes the proof.	
	\end{proof}
	
\section{A necessary condition for the splitting problem}\label{S3}
	In this section, we provide a necessary condition for the short exact sequence \eqref{ms} to split. Let $N$ be a normal subgroup of $B_{n,m}(S_g)$ that is also contained in $\beta_{n,m}$. Consider the following diagram of short exact sequences:
	\begin{equation}\label{cd}
		\begin{tikzcd}
			1 \ar[r] & \beta_{n,m} \ar[d,"pr|_{\beta_{n,m}}"] \ar[r] & B_{n,m}(S_g) \ar[d,"pr"] \ar[r, "\bar{q}"] & B_m(S_g) \ar[-,double line with arrow={-,-}]{d} \ar[r] & 1 \\
			1 \ar[r] & \beta_{n,m}/N \ar[r] & B_{n,m}(S_g)/N \ar[r, "q_*"] & B_m(S_g) \ar[r] & 1,
		\end{tikzcd}
	\end{equation}
	where $q_*:B_{n,m}(S_g)/N\to B_m(S_g)$ denotes the homomorphism induced by $\bar{q}$ and $pr$ is the canonical projection from $B_{n,m}(S_g)$ to $B_{n,m}(S_g)/N$.
	Suppose that the homomorphism $\bar{q}:B_{n,m}(S_g)\to B_m(S_g)$ admits a section $\bar{s}:B_m(S_g)\to B_{n,m}(S_g)$. Then $q_*:B_{n,m}(S_g)/N\to B_m(S_g)$ admits a section $s_*=pr\circ\bar{s}:B_m(S_g)\to B_{n,m}(S_g)/N$.
	For $N=\Gamma$, where $\Gamma:=\Gamma_2(\beta_{n,m})$ we obtain the following short exact sequence from \eqref{cd}:
	\begin{equation}\label{qfinal}
		\begin{tikzcd}
			1 \ar[r]& \beta_{n,m}/\Gamma \ar[r] & B_{n,m}(S_g)/\Gamma \ar[r, "q"] & B_m(S_g)\ar[r] & 1.
		\end{tikzcd}
	\end{equation}
	
	In order to study the sections of short exact sequence \eqref{qfinal} we need a presentation of the quotients $\beta_{n,m}/\Gamma$ and $B_{n,m}(S_g)/\Gamma$.
	We already  have a presentation of the Abelian group $\beta_{n,m}/\Gamma$, given by Corollary \ref{quot1}. We recall that $\beta_{n,m}/\Gamma\cong\mathbb{Z}^{2g+(m-1)} \times \mathbb{Z}_2$, where $a_1,\dots ,a_g$, $b_1,\dots ,b_g$, $z_1, \dots, z_{m-1}$ generate the $\mathbb{Z}^{2g+(m-1)}$-component and $\sigma$ the $\mathbb{Z}_2$-component. One can obtain a presentation of $B_{n,m}(S_g)/\Gamma$ using the presentation of $\beta_{n,m}/\Gamma$, given by Corollary \ref{quot1}, the presentation of $B_m(S_g)$ given in Theorem \ref{pres1}, and once again applying standard techniques for obtaining a presentation of group extensions as described in [\cite{johnsonpresentation}, p. 139] as presented in the proof of Theorem \ref{pres3}. However, since the group $\beta_{n,m}/\Gamma$ is the Abelianisation of $\beta_{n,m}$, from the commutative diagram of short exact sequences \eqref{cd}, a presentation of $B_{n,m}(S_g)/\Gamma$ may also be obtained straightforwardly. That is, one needs to consider the union of the generators of $\beta_{n,m}/\Gamma$ and the coset representatives of the generators of $B_n(S_g)$ as a set of generators, and the relations of $B_{n,m}(S_g)$, given in Theorem \ref{pres3}, projected into $B_{n,m}(S_g)/\Gamma$ as a set of relations, as described in the following proposition.

	\begin{proposition}\label{quot2}
	For $n, m\geq 1$, the following constitutes a presentation of $B_{n,m}(S_g)/\Gamma$, where $\Gamma=\Gamma_2(\beta_{n,m})$.
	\\
	\\ \textbf{Generators}: $\sigma$, $a_1, \dots, a_g$, $b_1, \dots, b_g$, $z_1,\dots,z_m$, $\tau_1,\dots ,\tau_{m-1}$, $c_1, \dots, c_g$, $d_1, \dots, d_g$.
	\\\textbf{Relations:} \begin{enumerate}	
		\item[$(I)$] Relations $(1)-(3)$ of Corollary \ref{quot1}.
		\item [$(II)$] \begin{enumerate}
			\item[$(a)$]Relations $(\overline{BR}), (\overline{R_1}), (\overline{R_2}), (\overline{R_3}), (\overline{R_4})$ of Theorem \ref{pres3}.
			\item[$(b)$] $[c_1,d_1^{-1}]\cdots[c_g,d_g^{-1}]\tau_1^{-1}\tau_2^{-1}\cdots\tau_{m-1}^{-1}\cdots\tau_2^{-1}\tau_1^{-1}=z_1^{-n}$.
		\end{enumerate}
		\item[$(III)$] For $1\leq i, k\leq m-1$, $1\leq r,s\leq g$,
		\begin{enumerate}
			\item[$(a)$] 
			$\tau_i, c_r, d_r$ pairwise commute with $\sigma$,\\
			$\tau_i, c_r$ pairwise commute with $a_s$,\\
			$\tau_i, d_r$ pairwise commute with $b_s$,\\
			$c_r, d_r$ pairwise commute with $z_k$,
			\item[$(b)$] 
			\begin{equation*}
				c_rb_sc_r^{-1}=
				\begin{cases}
					b_s, & \text{for}\ r\neq s,\\
					z_1^{-1}b_s, & \text{for} \ r=s,
				\end{cases}
			\end{equation*}	
			\item[$(c)$] 
			\begin{equation*}
				d_ra_sd_r^{-1}=
				\begin{cases}
					a_s, & \text{for}\ r\neq s,\\
					z_1a_s, & \text{for} \ r=s,
				\end{cases}
			\end{equation*}	
			\item[$(d)$] 
			\begin{equation*}
				\tau_iz_k\tau_i^{-1}=
				\begin{cases}
					z_k, & \text{for}\ i\neq k-1,k,\\
					z_{k-1}, & \text{for} \ i=k-1,\\
					z_{k+1}, & \text{for} \ i=k.
				\end{cases}
			\end{equation*}
		\end{enumerate}
	\end{enumerate}
	
\end{proposition}

\begin{remark}
	In Proposition \ref{quot2}, we denote by $\tau_1,\dots,\tau_{m-1}$, $c_1,\dots, c_g$ and $d_1,\dots, d_g$ the coset representatives of the generators $\sigma_1,\dots,\sigma_{m-1}$, $a_1,\dots, a_g$ and $b_1,\dots, b_g$ of $B_m(S_g)$, given by Theorem \ref{pres1}, in $B_{n,m}(S_g)/\Gamma$.
\end{remark}          

	Now, suppose that there exists a section $\bar{s} : B_m(S_g) \to B_{n,m}(S_g)$ for the homomorphism $\bar{q}: B_{n,m}(S_g) \to B_m(S_g)$. As we already saw, it follows that there exists a section $s_*:B_m(S_g) \to B_{n,m}(S_g)/\Gamma$ for $q_*:B_{n,m}(S_g)/\Gamma\to B_m(S_g)$. From Corollary \ref{quot1}, the set $\{a_1,\dots ,a_g, b_1,\dots, b_g, z_1,\dots,z_{m-1}, \sigma\}$ generates $\ker(q_*)$, which is the group $\beta_{n,m}/\Gamma\cong \mathbb{Z}^{2g+(m-1)}\times \mathbb{Z}_2$. This allows us to describe the image of the elements of the generating set of $B_m(S_g)$, under the section $s_*$, as follows:
	\begin{equation}\label{imageoft}
		s_*(\tau_i) =\tau_i\cdot a_1^{k_{i,1}}a_2^{k_{i,2}}\cdots a_g^{k_{i,g}}b_1^{l_{i,1}}b_2^{l_{i,2}}\cdots b_g^{l_{i,g}}
		z_1^{m_{i,1}}z_2^{m_{i,2}}\cdots z_{m-1}^{m_{i,{m-1}}}\sigma^{n_i},\ \text{for}\ 1\leq i \leq m-1,
	\end{equation}
	\begin{equation}\label{imageofc}
		s_*(c_r) =c_r\cdot a_1^{\bar{k}_{r,1}}a_2^{\bar{k}_{r,2}}\cdots a_g^{\bar{k}_{r,g}}b_1^{\bar{l}_{r,1}}b_2^{\bar{l}_{r,2}}\cdots b_g^{\bar{l}_{r,g}}
		z_1^{\bar{m}_{r,1}}z_2^{\bar{m}_{r,2}}\cdots z_{m-1}^{\bar{m}_{r,{m-1}}}\sigma^{\bar{n}_r},\ \text{for}\ 1\leq r \leq g,
	\end{equation} 
    \begin{equation}\label{imageofd}
	    s_*(d_r) =d_r\cdot a_1^{\tilde{k}_{r,1}}a_2^{\tilde{k}_{r,2}}\cdots a_g^{\tilde{k}_{r,g}}b_1^{\tilde{l}_{r,1}}b_2^{\tilde{l}_{r,2}}\cdots b_g^{\tilde{l}_{r,g}}
	z_1^{\tilde{m}_{r,1}}z_2^{\tilde{m}_{r,2}}\cdots z_{m-1}^{\tilde{m}_{r,{m-1}}}\sigma^{\tilde{n}_r},\ \text{for}\ 1\leq r \leq g,
    \end{equation} 
	where $k_{i,s}, l_{i,s}, m_{i,t}, \bar{k}_{r,s}, \bar{l}_{r,s}, \bar{m}_{r,t}, \tilde{k}_{r,s}, \tilde{l}_{r,s}, \tilde{m}_{r,t} \in \mathbb{Z}\ \text{and}\ n_i, \bar{n}_r, \tilde{n}_r\in\{0,1\}$, for every $1\leq i, t\leq m-1$ and $1\leq r, s\leq q$. Note that these integers are unique for every $s_*(\tau_i)$, $s_*(c_r)$ and $s_*(d_r)$, where $1\leq i \leq m-1$ and $1\leq r \leq g$.
	Under the assumption that there exists a section $s_*:B_m(S_g) \to B_{n,m}(S_g)/\Gamma$, the image of the relations, under $s_*$, in $B_m(S_g)$ are also relations in $B_{n,m}(S_g)/\Gamma$. In this way, we will obtain further information regarding the exponents in the formulas \eqref{imageoft}, \eqref{imageofc}, \eqref{imageofd}  and possible restrictions for the value of $n$, under the assumption that the short exact sequence \eqref{ms} splits. 
	
	Based on the presentation of $B_m(S_g)$ given by Theorem \ref{pres1}, we have the following six relations, which hold in $B_{n,m}(S_g)/\Gamma$:
	\begin{itemize}
			\item[\textbf{R1.}] $s_*(\tau_{i}\tau_j)=s_*(\tau_j\tau_i)$, for $|i-j|>1$,
		\item[] $s_*(\tau_i\tau_{i+1}\tau_i)=s_*(\tau_{i+1}\tau_i\tau_{i+1})$, for $1\leq i\leq m-2,$
		\item[\textbf{R2.}] $s_*(c_r\tau_{i})=s_*(\tau_{i}c_r)$, for $1\leq r\leq g$, $i\neq 1$,
		\item[] $s_*(d_r\tau_{i})=s_*(\tau_{i}d_r)$, for $1\leq r\leq g$, $i\neq 1$,
		\item[\textbf{R3.}] $s_*(\tau_1^{-1}c_r\tau_{1}^{-1}c_r)=s_*(c_r\tau_1^{-1}c_r\tau_{1}^{-1})$, for $1\leq r\leq g,$
		\item[] $s_*(\tau_1^{-1}d_r\tau_{1}^{-1}d_r)=s_*(d_r\tau_1^{-1}d_r\tau_{1}^{-1})$, for $1\leq r\leq g,$        
		\item[\textbf{R4.}] $s_*(\tau_1^{-1}c_s\tau_1c_r)=s_*(c_r\tau_1^{-1}c_s\tau_1)$, for $s < r,$
		\item[] $s_*(\tau_1^{-1}c_s\tau_1d_r)=s_*(d_r\tau_1^{-1}c_s\tau_1)$, for $s < r,$       
		\item[] $s_*(\tau_1^{-1}d_s\tau_1d_r)=s_*(d_r\tau_1^{-1}d_s\tau_1)$, for $s<r,$        
		\item[] $s_*(\tau_1^{-1}d_s\tau_1c_r)=s_*(c_r\tau_1^{-1}d_s\tau_1)$, for $s < r,$   
		\item[\textbf{R5.}] $s_*(\tau_1^{-1}c_r\tau_{1}^{-1}d_r)=s_*(d_r\tau_1^{-1}c_r\tau_{1})$, for $1\leq r\leq g$,
		\item[\textbf{R6.}] $s_*([c_1,d_1^{-1}]\cdots[c_g, d_g^{-1}])=s_*(\tau_{1}\tau_{2}\cdots\tau_{m-1}^2\cdots\tau_{2}\tau_{1})$.
	\end{itemize}
	In order to prove Theorem \ref{thm1}, we will make use of the following relations in $B_{n,m}(S_g)/\Gamma$ that appear in Proposition \ref{quot2}: \\
		For $1\leq i,k\leq m-1$ and $1\leq r, s\leq g$, we have the following set of relations :
	\begin{itemize}
		\item[$(S_1)$] $\sigma, a_1,\dots, a_g, b_1,\dots, b_g, z_1,\dots, z_m$ pairwise commute,
		\item[$(S_2)$] $\tau_i, c_r, d_r$ pairwise commute with $\sigma$,\\
		$\tau_i, c_r$ pairwise commute with $a_s$,\\
		$\tau_i, d_r$ pairwise commute with $b_s$,\\
		$c_r, d_r$ pairwise commute with $z_k$,
		\item[$(S_3)$] $\sigma^2=1$,
		\item[$(S_4)$] $z_m=(z_1\cdots z_{m-1})^{-1}$, 
		\item[$(S_5)$] $[c_1, d_1^{-1}]\cdots[c_g, d_g^{-1}]\tau_{1}^{-1}\tau_{2}^{-1}\cdots\tau_{m-1}^{-2}\cdots\tau_{2}^{-1}\tau_{1}^{-1}=z_1^{-n}$,
		\item[$(S_6)$] $b_sc_s=c_sb_sz_1$, 
		\item[$(S_7)$] $a_sd_s=d_sa_sz_1^{-1}$, 
		\item[$(S_8)$] $\tau_iz_{i+1}=z_i\tau_i$ and $\tau_iz_i=z_{i+1}\tau_i$.
	\end{itemize}
	\begin{proposition}\label{prop1}
		Let $g\geq2$, $n\geq 1$ and $m\geq 4$.
		If the following short exact sequence
		\begin{equation}\label{new}
			\begin{tikzcd}
				1 \ar[r]& B_n(S_g\setminus\{x_1,\dots,x_m\}) \ar[r] & B_{n,m}(S_g) \ar[r, "\bar{q}"] & B_m(S_g)\ar[r] & 1
			\end{tikzcd}
		\end{equation}
		splits, then $n=km + k(2g-2)$, where $k\in \mathbb{N}$.
	\end{proposition}
	
	\begin{proof}
		Based on the above discussion and on the assumption that the short exact sequence \eqref{new} splits, we have that the short exact sequence \eqref{qfinal} also splits. Thus, we will examine the relations \textbf{R1-R6}, which hold in $B_{n,m}(\mathbb{R}P^2)/\Gamma$, from which we will deduce that $n$ is a multiple of $m+(2g-2)$.
		
		We start with the relation $s_*(\tau_i\tau_{i+1}\tau_i)=s_*(\tau_{i+1}\tau_i\tau_{i+1})$, for $1\leq i\leq m-2$, of \textbf{R1}. 
		To treat this equality, we will use the formula \eqref{imageoft} and the relation $(S_8)$. To simplify the calculations, in the expression of $s_*(\tau_i\tau_{i+1}\tau_i)$ and $s_*(\tau_{i+1}\tau_i\tau_{i+1})$, and in all the expressions that follow, we keep only the elements that do not commute with all the other elements. This is why we will use the symbol $\approx$ instead of $=$. We can indeed ignore all the elements that commute with all the other elements since in each equality we can directly compare their exponents, since they are not affected by any possible change of position.
		Let $1\leq i\leq m-3$. We will treat the case $i=m-2$ separately.
		\begin{flalign*}
			s_*(\tau_{i}\tau_{i+1}\tau_{i}) &\approx 
			\tau_iz_1^{m_{i,1}}\cdots z_{m-1}^{m_{i,{m-1}}}\tau_{i+1} z_1^{m_{i+1,1}}\cdots z_{m-1}^{m_{i+1,{m-1}}}
			\tau_iz_1^{m_{i,1}}\cdots z_{m-1}^{m_{i,{m-1}}}&\\ 
			&\approx 
			\tau_i\tau_{i+1}z_1^{m_{i,1}}\cdots z_{i+2}^{m_{i,{i+1}}}z_{i+1}^{m_{i,{i+2}}}\cdots z_{m-1}^{m_{i,{m-1}}}
			\tau_{i}z_1^{m_{i+1,1}}\cdots z_{i+1}^{m_{i+1,i}}z_{i}^{m_{i+1,{i+1}}}\cdots z_{m-1}^{m_{i+1,{m-1}}}&\\
			&\quad z_1^{m_{i,1}}\cdots z_{m-1}^{m_{i,{m-1}}}&\\
			&\approx\tau_i\tau_{i+1}\tau_{i}z_1^{m_{i,1}}\cdots z_{i+1}^{m_{i,i}}z_{i+2}^{m_{i,{i+1}}}z_i^{i,i+2}\cdots z_{m-1}^{m_{i,{m-1}}}&\\
			&\quad z_1^{m_{i+1,1}}\cdots z_{i+1}^{m_{i+1,i}}z_{i}^{m_{i+1,{i+1}}}z_{i+2}^{i+1,i+2}\cdots z_{m-1}^{m_{i+1,{m-1}}} 
			z_1^{m_{i,1}}\cdots z_{i}^{m_{i,i}}z_{i+1}^{m_{i,{i+1}}}z_{i+2}^{i,i+2}\cdots z_{m-1}^{m_{i,{m-1}}}&\\
			&\approx\tau_i\tau_{i+1}\tau_{i}z_1^{2m_{i,1}+m_{i+1,1}}\cdots z_{i}^{m_{i,i+2}+m_{i+1,i+1}+m_{i,i}}
			z_{i+1}^{m_{i,{i}}+m_{i+1,i}+m_{i,i+1}}&\\
			&\quad z_{i+2}^{m_{i,{i+1}}+m_{i+1,i+2}+m_{i,i+2}}\cdots z_{m-1}^{2m_{i,{m-1}}+m_{i+1,m-1}}
		\end{flalign*}
		and similarly
		\begin{flalign*}
			s_*(\tau_{i+1}\tau_i\tau_{i+1}) &\approx\tau_{i+1}\tau_{i}\tau_{i+1}z_1^{2m_{i+1,1}+m_{i,1}}\cdots z_{i}^{m_{i+1,i+1}+m_{i,i}+m_{i+1,i}}
			z_{i+1}^{m_{i+1,{i+2}}+m_{i,i+2}+m_{i+1,i+1}}&\\
			&\quad z_{i+2}^{m_{i+1,{i}}+m_{i,i+1}+m_{i+1,i+2}}\cdots z_{m-1}^{2m_{i+1,{m-1}}+m_{i,m-1}}.
		\end{flalign*}
		Comparing the exponents of the elements $z_i, z_{i+1}, z_{i+2}$ and $z_j$, for $j\neq i, i+1, i+2$, of $B_{n,m}(S_g)/\Gamma$ in these two equations, we obtain the following:
		\begin{equation}\label{one}
			m_{i,j}=m_{i+1,j},\ \text{for}\ 1\leq  i\leq m-3\ \text{and}\ j\neq i, i+1, i+2.
		\end{equation}
		\begin{equation}\label{two}
			m_{i,i}+m_{i,i+1}=m_{i+1,i+1}+m_{i+1,i+2},\ \text{for}\ 1\leq i \leq m-3.
		\end{equation}
	Moreover, from the equality $s_*(\tau_i\tau_{i+1}\tau_i)=s_*(\tau_{i+1}\tau_i\tau_{i+1})$, for $1\leq i\leq m-2$, and from $(S_1)$, $(S_2)$, that is, $a_1,\dots, a_g, b_1,\dots, b_g, \sigma$ commute with all the other elements in the equality, comparing the exponents of these elements, we get that:
		\begin{align}	\label{three}
		&k_{i,j}=k_{i+1,j},\ \text{for}\ 1\leq  i\leq m-3\ \text{and}\ 1\leq j\leq g,&\\
		&l_{i,j}=l_{i+1,j},\ \text{for}\ 1\leq  i\leq m-3\ \text{and}\ 1\leq j\leq g,\nonumber&\\
		&n_i=n_{i+1},\ \text{for}\ 1\leq  i\leq m-3.\nonumber
		\end{align}
		In the case where $i=m-2$ we have the following:
			\begin{flalign*}
			s_*(\tau_{m-2}\tau_{m-1}\tau_{m-2}) &\approx 
			\tau_{m-2}z_1^{m_{m-2,1}}\cdots z_{m-2}^{m_{m-2,{m-2}}}z_{m-1}^{m_{m-2,{m-1}}}\tau_{m-1} z_1^{m_{m-1,1}}\cdots z_{m-2}^{m_{m-1,{m-2}}}z_{m-1}^{m_{m-1,{m-1}}}&\\
			&\quad \tau_{m-2}z_1^{m_{m-2,1}}\cdots z_{m-2}^{m_{m-2,{m-2}}}z_{m-1}^{m_{m-2,{m-1}}}&\\ 
			&\approx\tau_{m-2}\tau_{m-1}\tau_{m-2}z_1^{2m_{m-2,1}+m_{m-1,1}}\cdots
			z_j^{2m_{m-2,j}+m_{m-1,j}}\cdots 
			z_{m-2}^{m_{m-1,{m-1}}+m_{m-2,m-2}}&\\
			&\quad z_{m-1}^{m_{m-2,m-2}+m_{m-1,m-2}+m_{m-2,m-1}}
			 z_m^{m_{m-2,m-1}}
		\end{flalign*}
		and similarly
		\begin{flalign*}
			s_*(\tau_{m-1}\tau_{m-2}\tau_{m-1}) &\approx\tau_{m-1}\tau_{m-2}\tau_{m-1}z_1^{2m_{m-1,1}+m_{m-2,1}}\cdots
			z_j^{2m_{m-1,j}+m_{n-2,j}}\cdots &\\
			&\quad z_{m-2}^{m_{m-1,{m-1}}+m_{m-2,m-2}+m_{m-1, m-2}}
			z_{m-1}^{m_{m-1,m-1}}
			z_m^{m_{m-1, m-2}+m_{m-2,m-1}}.
		\end{flalign*}
		Note that the exponent of the element $z_m$ contributes as an exponent to all the other $z_j$, for $1\leq j\leq m-1$, due to relation $(S_4)$. Comparing the exponents of the elements $z_{m-2}, z_{m-1}$ and $z_j$, for $1\leq j\leq m-3$, of $B_{n,m}(S_g)/\Gamma$ in these two equations, we obtain the following:
		\begin{equation}\label{four}
			m_{m-2,j}=m_{m-1,j}-m_{m-1, m-2},\ \text{for}\ 1\leq  j\leq m-3.
		\end{equation}
		\begin{equation}\label{five}
			m_{m-2, m-2}-m_{m-1,m-1}= -(2m_{m-1, m-2}+m_{m-2, m-1}).
		\end{equation}
Moreover, from the equality $s_*(\tau_{m-2}\tau_{m-1}\tau_{m-2})=s_*(\tau_{m-1}\tau_{m-2}\tau_{m-1})$ and from the fact that $a_1,\dots, a_g, b_1,\dots, b_g, \sigma$ commute with all the other elements in the equality, comparing their exponents, we get that:
\begin{align}	\label{six}
	&k_{m-1,j}=k_{m-2,j},\ \text{for}\ 1\leq j\leq g,&\\
	&l_{m-1,j}=l_{m-2,j},\ \text{for}\ 1\leq j\leq g,\nonumber&\\
	&n_{m-1}=n_{m-2},\ \text{for}\ 1\leq  i\leq m-3.\nonumber
\end{align}
	
		We continue with the second relation of \textbf{R1}, $s_*(\tau_i\tau_j)=s_*(\tau_j\tau_i)$, for $i,j\neq m-1$ and $|i-j|\geq 2$. 
		We will discuss the case $i = m-1$ separately. Making use of $(S_8)$ we have the following:
		\begin{flalign*}
		s_*(\tau_i\tau_j)&\approx 
		\tau_i z_1^{m_i,1}\cdots z_{i-1}^{m_{i,i-1}}z_i^{m_{i,i}}z_{i+1}^{m_{i,i+1}}
		\cdots z_{j-1}^{m_{i,j-1}}z_j^{m_{i,j}}z_{j+1}^{m_{i,j+1}}\cdots z_{m-1}^{m_{i, m-1}}&\\
		&\quad\tau_jz_1^{m_{j,1}}z_1^{m_j,1}\cdots z_{i-1}^{m_{j,i-1}}z_i^{m_{j,i}}z_{i+1}^{m_{j,i+1}}
		\cdots z_{j-1}^{m_{j,j-1}}z_j^{m_{j,j}}z_{j+1}^{m_{j,j+1}}\cdots z_{m-1}^{m_{j, m-1}}&\\
		&\approx \tau_i\tau_j z_1^{m_{i,1}+m_{j,1}}\cdots z_{i-1}^{m_{i,i-1}+m_{j,i-1}}z_i^{m_{i,i}+m_{j,i}}z_{i+1}^{m_{i,i+1}+m_{j,i+1}}
		\cdots z_{j-1}^{m_{i,j-1}+m_{j,j-1}}&\\
		&\quad z_j^{m_{i,j+1}+m_{j,j}}z_{j+1}^{m_{i,j}+m_{j,j+1}}\cdots z_{m-1}^{m_{i, m-1}+m_{j,m-1}}
    \end{flalign*}
and similarly
	\begin{flalign*}
	s_*(\tau_j\tau_i)&\approx 
	\tau_jz_1^{m_{j,1}}z_1^{m_j,1}\cdots z_{i-1}^{m_{j,i-1}}z_i^{m_{j,i}}z_{i+1}^{m_{j,i+1}}
	\cdots z_{j-1}^{m_{j,j-1}}z_j^{m_{j,j}}z_{j+1}^{m_{j,j+1}}\cdots z_{m-1}^{m_{j, m-1}}&\\
	&\quad \tau_i z_1^{m_i,1}\cdots z_{i-1}^{m_{i,i-1}}z_i^{m_{i,i}}z_{i+1}^{m_{i,i+1}}
	\cdots z_{j-1}^{m_{i,j-1}}z_j^{m_{i,j}}z_{j+1}^{m_{i,j+1}}\cdots z_{m-1}^{m_{i, m-1}}&\\
	&\approx \tau_j\tau_i z_1^{m_{j,1}+m_{i,1}}\cdots z_{i-1}^{m_{j,i-1}+m_{i,i-1}}z_i^{m_{j,i+1}+m_{i,i}}z_{i+1}^{m_{j,i}+m_{i,i+1}}
	\cdots z_{j-1}^{m_{j,j-1}+m_{i,j-1}}&\\
	&\quad z_j^{m_{j,j}+m_{i,j}}z_{j+1}^{m_{j,j+1}+m_{i,j+1}}\cdots z_{m-1}^{m_{j, m-1}+m_{i,m-1}}.
\end{flalign*}
		Comparing the exponents of the elements $z_i$ and $z_{i+1}$, for $i,j\neq m-1$ and $|i-j|\geq 2$, of $B_{n,m}(S_g)/\Gamma$ in these two equations, we obtain the following:
	\begin{equation}\label{onep5}
		m_{j,i}=m_{j,i+1},\ \text{for}\ 1\leq  i,j\leq m-2\ \text{and}\ |i-j|\geq 2.
	\end{equation}
		In the case where $i=m-1$ and $|i-j|\geq 2$ we have the following:
		\begin{flalign*}
			s_*(\tau_{m-1}\tau_j)
			&\approx \tau_{m-1}\tau_j z_1^{m_{m-1,1}+m_{j,1}}\cdots  z_j^{m_{m-1,j+1}+m_{j,j}}z_{j+1}^{m_{m-1,j}+m_{j,j+1}}\cdots z_{m-1}^{m_{m-1, m-1}+m_{j,m-1}}
		\end{flalign*}
		and similarly
	\begin{flalign*}
		s_*(\tau_{j}\tau_{m-1})
		&\approx\tau_j\tau_{m-1} z_1^{m_{j,1}+m_{m-1,1}}\cdots  z_j^{m_{j,j}+m_{m-1,j}}z_{j+1}^{m_{j,j+1}+m_{m-1,j+1}}\cdots z_{m-1}^{m_{m-1,m-1}}z_m^{m_{j,m-1}}.
	\end{flalign*}
		Note that the exponent of the element $z_m$ contributes as an exponent to all the other $z_j$, for $1\leq j\leq m-1$, due to relation $(S_4)$. Comparing the exponents of the elements $z_{j}$ and $z_{j+1}$, for $1\leq j\leq m-3$, of $B_{n,m}(S_g)/\Gamma$ in these two equations, we obtain the following:
		\begin{equation}\label{twop5}
			m_{j,m-1}=0,\ \text{for}\ 1\leq  j\leq m-3.
		\end{equation}
		\begin{equation}\label{threep5}
			m_{m-1, j}=m_{m-1, j+1},\ \text{for}\ 1\leq  j\leq m-3.
		\end{equation}
	Thus, from the obtained relations $\eqref{three}$ and $\eqref{six}$ we set $$k_j:=k_{i,j},\ l_j:=l_{i,j},\ N:=n_i,\ \text{for all}\ 1\leq i\leq m-1\ \text{and}\ 1\leq j\leq g.$$
	It follows that the relation \ref{imageoft} becomes 
\begin{equation}\label{imageoftnew}
s_*(\tau_i) =\tau_i\cdot a_1^{k_1}a_2^{k_2}\cdots a_g^{k_g}b_1^{l_1}b_2^{l_2}\cdots b_g^{l_g}
z_1^{m_{i,1}}z_2^{m_{i,2}}\cdots z_{m-1}^{m_{i,{m-1}}}\sigma^{N},\ \text{for}\ 1\leq i \leq m-1.
\end{equation}

		We continue with the relation $s_*(c_r\tau_i) = s_*(\tau_ic_r)$, for $1< i< m-1$ and $1\leq r\leq g$, of \textbf{R2}.
		We will discuss the case $i = m-1$ separately. So, for $1< i < n-2$, based on the formulas \eqref{imageoftnew}, \eqref{imageofc} and the relations $(S_6)$, $(S_8)$ we have the following:
		\begin{flalign*}
			s_*(c_r\tau_i)
			&\approx c_r\tau_ib_1^{\bar{l}_{r,1}+l_1}\cdots b_g^{\bar{l}_{r,g}+l_g}&\\
			&\quad z_1^{\bar{m}_{r,1}+m_{i,1}}\cdots z_i^{\bar{m}_{r,i+1}+m_{i,i}}
			z_{i+1}^{\bar{m}_{r,i}+m_{i,i+1}}\cdots z_{m-1}^{\bar{m}_{r,m-1}+m_{i,m-1}}
			\sigma^{\bar{n}_r+N}
		\end{flalign*}
		and
		\begin{flalign*}
			s_*(\tau_{i}c_r)
		&\approx\tau_ic_rb_1^{l_1+\bar{l}_{r,1}}\cdots b_g^{l_g+\bar{l}_{r,g}}&\\
		&\quad z_1^{m_{i,1}+l_r+\bar{m}_{r,1}}\cdots z_i^{m_{i,i}+\bar{m}_{r,i}}
		z_{i+1}^{m_{i,i+1}+\bar{m}_{r,i+1}}\cdots z_{m-1}^{m_{i,m-1}+\bar{m}_{r,m-1}}
		\sigma^{N+\bar{n}_r}.
		\end{flalign*}
		Comparing the coefficients of the element $z_1$, $z_i$ and $z_{i+1}$, for $1< i\leq m-2$ of $B_{n,m}(S_g)/\Gamma$ in these two equations, it follows that 
		\begin{equation}\label{onep6}
			l_r=0,\ \text{for}\ 1\leq  r\leq g
		\end{equation}
	and
	\begin{equation}\label{twop6}
		\bar{m}_{r,i}=\bar{m}_{r,i+1},\ \text{for}\ 1< i\leq m-2\ \text{and}\ 1\leq r\leq g.
	\end{equation}
		In the case $i=m-1$ we have the following:
		\begin{flalign*}
			s_*(c_r\tau_{m-1})
			&\approx c_r\tau_{m-1}b_1^{\bar{l}_{r,1}+l_1}\cdots b_g^{\bar{l}_{r,g}+l_g}&\\
			&\quad z_1^{\bar{m}_{r,1}+m_{m-1,1}}\cdots z_i^{\bar{m}_{r,i}+m_{m-1,i}}\cdots z_{m-1}^{m_{m-1,m-1}}z_{m}^{\bar{m}_{r,m-1}}
			\sigma^{\bar{n}_r+N}
		\end{flalign*}
		and
		\begin{flalign*}
			s_*(\tau_{m-1}c_r)
			&\approx\tau_{m-1}c_rb_1^{l_1+\bar{l}_{r,1}}\cdots b_g^{l_g+\bar{l}_{r,g}}&\\
			&\quad z_1^{m_{m-1,1}+\bar{m}_{r,1}}\cdots z_i^{m_{m-1,i}+\bar{m}_{r,i}}
			\cdots z_{m-1}^{m_{m-1,m-1}+\bar{m}_{r,m-1}}
			\sigma^{N+\bar{n}_r}.
		\end{flalign*}
		Note that the exponent of the element $z_m$ contributes as an exponent to all the other $z_j$, for $1\leq j\leq m-1$, due to relation $(S_4)$.
		Comparing the coefficients of the element $z_1$ , it follows that 
		\begin{equation}\label{threep6}
			\bar{m}_{r,m-1}=0,\ \text{for}\ 1\leq  r\leq g.
		\end{equation}
		
		Similarly, from the relation of \textbf{R2}, $s_*(d_r\tau_i) = s_*(\tau_id_r)$, for $1< i< m-1$ and $1\leq r\leq g$, and with the use of the relation $(S_7)$, $(S_8)$ we get that 
			\begin{equation}\label{fourp7}
			k_r=0,\ \text{for}\ 1\leq  r\leq g.
		\end{equation}
			and
		\begin{equation}\label{fivep7}
			\tilde{m}_{r,i}=\tilde{m}_{r,i+1},\ \text{for}\ 1< i\leq m-2\ \text{and}\ 1\leq r\leq g.
		\end{equation}
		In addition, for $i=m-1$ from the relation $s_*(d_r\tau_{m-1}) = s_*(\tau_{m-1}d_r)$, for $1\leq r\leq g$, it follows that
		\begin{equation}\label{sixp7}
		\tilde{m}_{r,m-1}=0,\ \text{for}\ 1\leq  r\leq g.
	\end{equation}
		
	The calculation of the remaining relations \textbf{R3}, \textbf{R4} and \textbf{R5} is straightforward and we present directly the results one will obtain. However, the last relation \textbf{R6} will be presented in detail.
	
	From the relation of \textbf{R3}, $s_*(\tau_1^{-1}c_r\tau_{1}^{-1}c_r)=s_*(c_r\tau_1^{-1}c_r\tau_{1}^{-1})$, for $1\leq r\leq g,$ and using the relations $(S_6)$, $(S_8)$, comparing the exponents of the element $z_1$ we get that 
	\begin{equation}\label{onep8}
		\bar{l}_{r,r}=0,\ \text{for}\ 1\leq  r\leq g.
	\end{equation}
From the relation of \textbf{R3}, $s_*(\tau_1^{-1}d_r\tau_{1}^{-1}d_r)=s_*(d_r\tau_1^{-1}d_r\tau_{1}^{-1})$, for $1\leq r\leq g,$ and using the relations $(S_7)$, $(S_8)$, comparing the exponents of the element $z_1$ we obtain that 
\begin{equation}\label{twop8}
	\tilde{k}_{r,r}=0,\ \text{for}\ 1\leq  r\leq g.
\end{equation}

     From the relation of \textbf{R4}, $s_*(\tau_1^{-1}c_s\tau_1c_r)=s_*(c_r\tau_1^{-1}c_s\tau_1)$, for $s < r$,
     and using the relations $(S_6)$, $(S_8)$, comparing the exponents of the elements $z_1$ and $z_2$ we obtain that 
     \begin{equation}\label{onep9}
     	\bar{l}_{p,q}=0,\ \text{for}\ 1\leq  p\neq q\leq g.
     \end{equation}
	From the relation of \textbf{R4}, $s_*(\tau_1^{-1}c_s\tau_1d_r)=s_*(d_r\tau_1^{-1}c_s\tau_1)$, for $s < r,$ 
	and using the relations $(S_6)$, $(S_7)$, $(S_8)$, comparing the exponents of the elements $z_1$ and $z_2$ we obtain that 
	\begin{equation}\label{onep10}
		\bar{k}_{s,r}=0,\ \text{for}\ 1\leq  s<r\leq g,
	\end{equation}      
and
	\begin{equation}\label{twop10}
	\tilde{l}_{r,s}=0,\ \text{for}\ 1\leq  s<r\leq g.
\end{equation}   
	From the relation of \textbf{R4}, $s_*(\tau_1^{-1}d_s\tau_1d_r)=s_*(d_r\tau_1^{-1}d_s\tau_1)$, for $s<r,$ 
	and using the relations $(S_7)$, $(S_8)$, comparing the exponents of the elements $z_1$ and $z_2$ we obtain that 
	\begin{equation}\label{twop9}
		\tilde{k}_{p,q}=0,\ \text{for}\ 1\leq  p\neq q\leq g.
	\end{equation}       
	From the relation of \textbf{R4}, $s_*(\tau_1^{-1}d_s\tau_1c_r)=s_*(c_r\tau_1^{-1}d_s\tau_1)$, for $s < r,$   
	and using the relations $(S_6)$, $(S_7)$, $(S_8)$, comparing the exponents of the elements $z_1$ and $z_2$ we obtain that 
	\begin{equation}\label{one+p10}
		\tilde{l}_{s,r}=0,\ \text{for}\ 1\leq  s<r\leq g.
	\end{equation}      
	and
	\begin{equation}\label{two+p10}
		\bar{k}_{r,s}=0,\ \text{for}\ 1\leq  s<r\leq g.
	\end{equation}   
	
	From the relation of \textbf{R5}, $s_*(\tau_1^{-1}c_r\tau_{1}^{-1}d_r)=s_*(d_r\tau_1^{-1}c_r\tau_{1})$, for $1\leq r\leq g$,
	and using the relations $(S_6)$, $(S_7)$, $(S_8)$, comparing the exponents of the elements $z_1, z_2$ and $z_i$, for $3\leq i\leq m-1$ we obtain that 
	\begin{equation}\label{onep11}
		\bar{k}_{r,r}=-(m_{1,1}+m_{1,2}),\ \text{for}\ 1\leq  r\leq g,
	\end{equation}	
		\begin{equation}\label{twop11}
\tilde{l}_{r,r}=-(m_{1,1}+m_{1,2}),\ \text{for}\ 1\leq  r\leq g
		\end{equation}
	and
		\begin{equation}\label{threep11}
			m_{1,i}=0,\ \text{for}\ 3\leq  i\leq m-1.
		\end{equation}
	
	We now analyse the results that we obtained so far. From \eqref{onep6}, \eqref{fourp7} we obtain $l_r=0$ and $k_r=0$, for all $1\leq r\leq g$. From \eqref{onep6}, \eqref{twop5}, \eqref{threep11} it follows that $m_{i,j}=0$, for $1\leq i,j\leq m-2$ and $i+2\leq j$. In addition, from \eqref{one}, \eqref{four}, \eqref{threep5} it follows that $m_{i,j}=0$, for $2\leq i,j\leq m-2$ and $j<i$. All together we get that $m_{i,j}=0$, for $1\leq i,j\leq m-2$ and $j\neq i, i+1$. That it, $m_{i,i}\neq 0$ and $m_{i,i+1}\neq 0$, for $1\leq i\leq m-2$. From \eqref{onep8}, \eqref{onep9} we have $\bar{l}_{r,s}=0$, for all $1\leq r,s\leq g$. Moreover, from \eqref{onep10}, \eqref{two+p10} we get $\bar{k}_{r,s}$, for $1\leq r\neq s\leq g$. From \eqref{twop8}, \eqref{twop9} we get $\tilde{k}_{r,s}=0$, for all $1\leq r, s\leq g$ and from \eqref{twop10}, \eqref{one+p10} we have $\tilde{l}_{r,s}=0$, for $1\leq r\neq s\leq g$. From \eqref{twop6}, \eqref{threep6} it follows that $\bar{m}_{r,i}=0$, for $1<i\leq m-1$ and $1\leq r\leq g$. Similarly from \eqref{fivep7}, \eqref{sixp7} it follows that $\tilde{m}_{r,i}=0$, for $1<i\leq m-1$ and $1\leq r\leq g$. From \eqref{threep5} it holds that $m_{m-1,1}=m_{m-1,1}=\dots=m_{m-1,m-2}$, which integer we will denote by $M$. From \eqref{onep11}, \eqref{twop11} we have $\bar{k}_{r,r}=\tilde{l}_{r,r}=-(m_{1,1}+m_{1,2})$, which integer we will denote by $k$. For simplicity, we will now denote $\bar{m}_{r,1}$ by $\bar{m}_r$, for $1\leq r\leq g$ and similarly, $\tilde{m}_{r,1}$ by $\tilde{m}_r$, for $1\leq r\leq g$. Lastly, from \eqref{two}, \eqref{five} we obtain that $m_{1,1}+m_{1,2}=m_{1,1}+m_{1,2}=\dots=m_{m-2,m-2}+m_{m-2,m-1}=m_{m-1,m-1}-2M$.
	
Thus, combining all these results, the relations \eqref{imageoftnew},\eqref{imageofc} and \eqref{imageofd} become as follows:
\begin{flalign*}
&s_*(\tau_i) =\tau_1z_i^{m_{i,i}}z_{i+1}^{m_{i,i+2}}\sigma^N,\ \text{for}\ 1\leq i \leq m-2,&\\
&s_*(\tau_{m-1}) =\tau_{m-1}z_1^Mz_2^M\cdots z_{m-2}^Mz_{m-1}^{m_{m-1,m-1}}\sigma^N,&\\
&s_*(c_r) =c_ra_r^kz_1^{\bar{m}_{r}}\sigma^{\bar{n}_r},\ \text{for}\ 1\leq r \leq g,&\\
&s_*(d_r) =d_rb_r^kz_1^{\tilde{m}_{r}}\sigma^{\tilde{n}_r},\ \text{for}\ 1\leq r \leq g,&\\
&\text{where,}&\\
&k:=-(m_{1,1}+m_{1,2}),&\\ 
&M:=m_{m-1,i},\ \text{for all}\ 1\leq i\leq m-2,&\\
&m_{i,i}+m_{i,i+1}=m_{m-1,m-1}-2M,\ \text{for all}\ 1\leq i\leq m-2.
\end{flalign*}
		
	Finally, with these new images of $\tau_i$, $c_r$ and $d_r$, under $s_*$, we will examine the relation \textbf{R6}, $s_*([c_1,d_1^{-1}]\cdots[c_g, d_g^{-1}])=s_*(\tau_{1}\tau_{2}\cdots\tau_{m-1}^2\cdots\tau_{2}\tau_{1})$. Note that $s_*(\tau_{1}\tau_{2}\cdots\tau_{m-1}^2\cdots\tau_{2}\tau_{1})=s_*(\tau_1\tau_2\cdots\tau_{m-2}\tau_{m-1})s_*(\tau_{m-1}\tau_{m-2}\cdots\tau_2\tau_1)$.
		We will compute the image of $s_*(\tau_1\tau_2\cdots\tau_{m-2}\tau_{m-1})$ and of $s_*(\tau_{m-1}\tau_{m-2}\cdots\tau_2\tau_1)$ separately, making use of the relation $(S_8)$.
		\begin{flalign*}
			s_*(\tau_1\tau_2\cdots\tau_{m-2}\tau_{m-1}) 
			&=\tau_1z_1^{m_{1,1}}z_2^{m_{1,2}}\sigma^N
			\tau_2z_2^{m_{2,2}}z_3^{m_{2,3}}\sigma^N\cdots\tau_iz_i^{m_{i,i}}z_{i+1}^{m_{i,i+1}}\sigma^N\cdots&\\
			&\quad\tau_{m-2}z_{m-2}^{m_{m-2,m-2}}z_{m-1}
			^{m_{m-2,m-1}}\sigma^N\quad t_{m-1}z_1^Mz_2^M\cdots z_{m-2}^Mz_{m-1}^{m-1,m-1}\sigma^N&\\
			&=\tau_{1}\tau_{2}\cdots\tau_{i}\cdots\tau_{m-1}&\\
			&=\tau_1\tau_2\cdots\tau_{m-2}\tau_{m-1}z_1^{M+m_{1,1}-\mu}
			z_2^{M+m_{2,2}-\mu}\cdots z_i^{M+m_{i,i}-\mu}\cdots z_{m-2}^{M+m_{m-2,m-2}-\mu}&\\
			&\quad z_{m-1}^{m_{m-1,m-1}-\mu}\sigma^{(m-1)N},&\\
\text{where}\ \mu:=m_{1,2}+&m_{2,3}+\dots+m_{i,i+1}+\dots+m_{m-2,m-1}.
							\end{flalign*}
Similarly,
		\begin{flalign*}
			s_*(\tau_{m-1}\tau_{m-2}\cdots\tau_2\tau_1)
			&= \tau_{m-1}\tau_{m-2}\cdots\tau_2\tau_1z_1^{m_{m-1,m-1}+m_{m-2,m-2}
			+\dots+m_{i,i}+\dots+m_{1,1}}&\\
		&\quad z_2^{M+m_{1,2}}\cdots z_i^{M+m_{i-1,i}}\cdots z_{m-2}^{M+m_{m-3,m-2}}z_{m-1}^{M+m_{m-1,m-1}}\sigma^{(m-1)N}.
		\end{flalign*}
		Therefore,
		\begin{flalign*}
			s_*(\tau_1\tau_2\cdots\tau_{m-2}\tau_{m-1}^2\tau_{m-2}\cdots\tau_2\tau_1) &=
			s_*(\tau_1\tau_2\cdots\tau_{m-2}\tau_{m-1}) 
			s_*(\tau_{m-1}\tau_{m-2}\cdots\tau_2\tau_1)&\\
			&=\tau_1\tau_2\cdots\tau_{m-2}\tau_{m-1}z_1^{M+m_{1,1}-\mu}
			z_2^{M+m_{2,2}-\mu}\cdots z_i^{M+m_{i,i}-\mu}\cdots&\\ &\quad z_{m-2}^{M+m_{m-2,m-2}-\mu}&\\
			&\quad z_{m-1}^{m_{m-1,m-1}-\mu}\sigma^{(m-1)N}&\\
			& \quad \tau_{m-1}\tau_{m-2}\cdots\tau_2\tau_1z_1^{m_{m-1,m-1}+m_{m-2,m-2}
				+\dots+m_{i,i}+\dots+m_{1,1}}&\\
			&\quad z_2^{M+m_{1,2}}\cdots z_i^{M+m_{i-1,i}}\cdots&\\ &\quad z_{m-2}^{M+m_{m-3,m-2}}z_{m-1}^{M+m_{m-1,m-1}}\sigma^{(m-1)N}&\\
			&=\tau_1\tau_2\cdots\tau_{m-2}\tau_{m-1}^2\tau_{m-2}\cdots\tau_2\tau_1&\\
			&\quad z_1^{m_{1,1}+\dots+m_{i,i}+\dots+m_{m-1,m-1}-m_{m-1,m-1}+\mu}&\\
			&\quad z_2^{2M+m_{1,1}+m_{1,2}-m_{m-1,m-1}+\mu-\mu}\cdots&\\
			 &\quad z_i^{2M+m_{i-1,i-1}+m_{i-1,i}-m_{m-1,m-1}+\mu-\mu}\cdots&\\
			&\quad z_{m-1}^{2M+m_{m-2,m-2}+m_{m-2,m-1}-m_{m-1,m-1}+\mu-\mu}&\\
			&=\tau_1\tau_2\cdots\tau_{m-2}\tau_{m-1}^2\tau_{m-2}\cdots\tau_2\tau_1z_1^{(m-2)(m_{m-1,m-1}-2M)}\cdots z_i^0\cdots z_{m-1}^0&\\
			&=\tau_1\tau_2\cdots\tau_{m-2}\tau_{m-1}^2\tau_{m-2}\cdots\tau_2\tau_1\cdot z_1^{(m-2)(m_{m-1,m-1}-2M)}.
		\end{flalign*}
		Moreover, with the use of the relations $(S_6)$, $(S_7)$ we have the following:
		\begin{flalign*}
			s_*([c_1, d_1^{-1}]\cdots[c_g, d_g^{-1}]) &= 
			s_*(c_1d_1^{-1}c_1^{-1}d_1\cdots c_rd_r^{-1}c_r^{-1}d_r\cdots c_gd_g^{-1}c_g^{-1}d_g)&\\
			&=c_1a_1^kz_1^{\bar{m}_1}\sigma^{\bar{n}_1}d_1^{-1}b_1^{-k}
			z_1^{-\tilde{m}_1}\sigma^{-\tilde{n}_1}c_1^{-1}
			a_1^{-k}z_1^{-\bar{m}_1}\sigma^{-\bar{n}_1}d_1b_1^{k}
			z_1^{\tilde{m}_1}\sigma^{\tilde{n}_1}\cdots&\\
			&\quad c_ra_r^kz_1^{\bar{m}_r}\sigma^{\bar{n}_r}d_r^{-1}b_r^{-k}
			z_1^{-\tilde{m}_r}\sigma^{-\tilde{n}_r}c_r^{-1}
			a_r^{-k}z_1^{-\bar{m}_r}\sigma^{-\bar{n}_r}d_rb_r^{k}
			z_1^{\tilde{m}_r}\sigma^{\tilde{n}_r}\cdots&\\
			&\quad c_ga_g^kz_1^{\bar{m}_g}\sigma^{\bar{n}_g}d_g^{-1}b_g^{-k}
			z_1^{-\tilde{m}_g}\sigma^{-\tilde{n}_g}c_g^{-1}
			a_g^{-k}z_1^{-\bar{m}_g}\sigma^{-\bar{n}_g}d_gb_g^{k}
			z_1^{\tilde{m}_g}\sigma^{\tilde{n}_g}&\\
			&=c_1d_1^{-1}c_1^{-1}d_1\cdots c_rd_r^{-1}c_r^{-1}d_r\cdots c_gd_g^{-1}c_g^{-1}d_gz_1^{2kg}&\\
			&=[c_1, d_1^{-1}]\cdots[c_g, d_g^{-1}]\cdot z_1^{2kg}.
		\end{flalign*}
		From $(S_5)$ we have that $[c_1, d_1^{-1}]\cdots[c_g, d_g^{-1}]\tau_{1}^{-1}\tau_{2}^{-1}\cdots\tau_{m-1}^{-2}\cdots\tau_{2}^{-1}\tau_{1}^{-1}=z_1^{-n}$ in $B_{n,m}(S_g)/\Gamma$. Therefore, 
		\begin{flalign*}
		[c_1, d_1^{-1}]\cdots[c_g,d_g^{-1}]&=
		z_1^{-n}\cdot\tau_1\tau_2\cdots\tau_{m-2}
		\tau_{m-1}^2\tau_{m-2}\cdots\tau_2\tau_1&\\
		&=\tau_1\tau_2\cdots\tau_{m-2}
		\tau_{m-1}^2\tau_{m-2}\cdots\tau_2\tau_1\cdot z_1^{-n}.
     	\end{flalign*}
		Thus, $s_*([c_1, d_1^{-1}]\cdots[c_g, d_g^{-1}])= \tau_1\tau_2\cdots\tau_{m-2}
		\tau_{m-1}^2\tau_{m-2}\cdots\tau_2\tau_1\cdot z_1^{-n}z_1^{2kg}$.
		
		Comparing the exponents of the element $z_1$ of the equality $s_*([c_1, d_1^{-1}]\cdots[c_g, d_g^{-1}])=	s_*(\tau_1\tau_2\cdots\tau_{m-2}\tau_{m-1}^2\tau_{m-2}\cdots\tau_2\tau_1)$ we obtain the following: $2kg-n=(m-2)(m_{m-1,m-1}-2M)$. That is, 
		\begin{flalign*}
		n&=2kg-(m-2)(m_{m_{m-1,m-1}}-2M)&\\
		&=2kg-(m-2)(-k)&\\
		&=2kg+k(m-2).
		\end{flalign*}
	
	We deduce that $n=km+k(2g-2)$, for $k\in \mathbb{N}$, which concludes the proof.
	\end{proof}

\begin{remark}\label{cases}
For $g\geq 2$ and $m=2$, using the the same strategy as in the proof of Proposition \ref{prop1} we obtain that $n=2kg$, for $k\in \mathbb{N}$. In particular, we examine the relations \textbf{R3}-\textbf{R6}, since the relations \textbf{R1}-\textbf{R2} do not hold for $m=2$.

For $g\geq 2$ and $m=3$, using the the same strategy as in the proof of Proposition \ref{prop1} we obtain that $n=k(2g+1)$, for $k\in \mathbb{N}$. In particular, we examine the relations \textbf{R1}-\textbf{R6}, except for the relation $s_*(\tau_{i}\tau_{j})=s_*(\tau_{j}\tau_{i})$ of \textbf{R1}, for $|i-j|>1$, which does not hold for $m=3$.

For $g=1$ and $m=2$, using the the same strategy as in the proof of Proposition \ref{prop1} we obtain that $n=2k$, for $k\in \mathbb{N}$. In particular, we examine the relations \textbf{R3}, \textbf{R5}, \textbf{R6}, since the relations \textbf{R1}, \textbf{R2}, \textbf{R4} do not hold for $g=1$ and $m=2$.

For $g=1$ and $m=3$, using the the same strategy as in the proof of Proposition \ref{prop1} we obtain that $n=3k$, for $k\in \mathbb{N}$. In particular, we examine the relations \textbf{R1}-\textbf{R6}, except for the relation $s_*(\tau_{i}\tau_{j})=s_*(\tau_{j}\tau_{i})$ of \textbf{R1}, for $|i-j|>1$, and the set of relations of \textbf{R4}, which does not hold for $g=1$ and $m=3$.

For $g=1$ and $m\geq 4$, using the the same strategy as in the proof of Proposition \ref{prop1} we obtain that $n=mk$, for $k\in \mathbb{N}$, examining the relations \textbf{R1}-\textbf{R6}, except for the set of relations of \textbf{R4}, which does not hold for $g=1$.
\end{remark}

In the following proposition we provide a geometric section from which we can deduce an algebraic section for $B_{n,1}(S_g)\to B_1(S_g)$, using a similar construction, given by Gonçalves--Guaschi in \cite{gonccalves2004structure},  the case of pure braid groups.

\begin{proposition}\label{prop2}
Let $g\geq 1$, $n\geq 1$ and $m=1$. The short exact sequence 
\begin{equation}
	\begin{tikzcd}
		1 \ar[r]& B_n(S_g\setminus\{x_1\}) \ar[r] & B_{n,1}(S_g) \ar[r, "\bar{q}"] & B_1(S_g)\ar[r] & 1,
	\end{tikzcd}
\end{equation}
splits for all values of $n$ and $g$.
\end{proposition}

\begin{proof}
Let $g\geq 1$, $n\geq 1$ and $m=1$. In order to provide an algebraic section for the map $\bar{q}: B_{n,1}(S_g)\to B_1(S_g)$ we will construct a geometric section in the level of configuration spaces, see Remark \ref{equiv}. That is, we will present a section for the map $q: UF_{n,1}(S_g)\to UF_1(S_g)$, defined by forgetting the first $n$ coordinates. We construct such a section in the following way.

 We known that there exists a circle $C$ on $S_g$ that is a retract of $S_g$. That is, for $i:C\to S_g$, the inclusion of $C$ into $S_g$, there exists a continuous map $r:S_g\to C$, such that, for every $x\in C$, $r(x)=x$ and $r\circ i=\text{id}_{C}$. We take $C$ to be any meridian of one of the handles of $S_g$. Thus, for such a choice of $C$ there exists indeed a retraction of $S_g$ onto $C$, since $C$ is a non-separating circle of $S_g$ as it is one of the meridians.
 
 Using this retraction, we will define $n$ pairwise coincidence-free continuous self maps of $S_g$. For $1=1,\dots, n$, we define $f_i:S_g\to S_g$ to be $R_i\circ r$, the composition of the retraction $r$ with the rotation $R_i$ of $C$ of angle $2\pi i/(n+1)$. With the rotation we ensure that for every point $x\in S_g$ it holds that $f_i(x)\neq f_j(x)$ and $f_i(x)\neq \text{id}_{S_g}(x)$, for all $i,j$ such that $1\leq i, j\leq n$ and $i\neq j$. 
 Thus, the following $n+1$ continuous self maps of $S_g$, $\text{id}_{S_g}, f_1, \dots, f_n$,  are pairwise coincidence-free.
 
 Thus, by setting $g=f_1\times\cdots\times f_n\times\text{id}_{S_g} $ we construct a continuous map from $UF_1(S_g)$ to $UF_{n,1}(S_g)$, that is indeed a geometric section since $q\circ g=\text{id}_{UF_1(S_g)}$. As a result, we deduce, for all $g\geq 1$ and $n\geq 1$, a section for the map $\bar{q}: B_{n,1}(S_g)\to B_1(S_g)$.

\end{proof}

\begin{proof}[Proof of Theorem \ref{thm1}]
From Proposition \ref{prop1} and Remark \ref{cases}, for $g\geq 1$, $n\geq 1$ and $m>1$, we have that, if there is a section for $B_{n,m}(S_g)\to B_m(S_g)$, then $n=km+k(2g-2)$, for $k\in\mathbb{N}$. Moreover, from Proposition \ref{prop2}, it holds that $B_{n,1}(S_g)\to B_1(S_g)$ admits a section for all values of $n\geq 1$ and $g\geq1$, which completes the proof.
\end{proof}

\section*{Acknowledgement}
The author gratefully acknowledges the Max Planck Institute for Mathematics in Bonn for its hospitality and financial support during a research stay, which enabled the initial phase of this research project.

\bibliographystyle{plain}
\bibliography{bibliSP}
	\end{document}